\documentclass[pdflatex,sn-mathphys-num]{sn-jnl}

\usepackage{graphicx}%
\usepackage{multirow}%
\usepackage{amsmath,amssymb,amsfonts}%
\usepackage{amsthm}%
\usepackage{mathrsfs}%
\usepackage[title]{appendix}%
\usepackage{xcolor}%
\usepackage{textcomp}%
\usepackage{manyfoot}%
\usepackage{booktabs}%
\usepackage{algorithm}%
\usepackage{algorithmicx}%
\usepackage{algpseudocode}%
\usepackage{listings}%
\usepackage{subcaption}
\usepackage{multirow}
\usepackage{overpic}

\theoremstyle{thmstyleone}%
\newtheorem{theorem}{Theorem}
\newtheorem{lemma}[theorem]{Lemma}

\theoremstyle{thmstyletwo}%
\newtheorem{example}{Example}%

\theoremstyle{thmstylethree}%
\newtheorem{definition}{Definition}%

\newcommand{\conc}{\mbox{conc}}
\newcommand{\inte}{\mbox{int}}
\newcommand{\cl}{\mbox{cl}}
\newcommand{\conv}{\mbox{conv}}

\begin{document}

\title{
Tightening convex relaxations of trained neural networks: a unified approach for convex and S-shaped activations
}

\author[1]{\fnm{Pablo} \sur{Carrasco}}\email{pablo.carrasco@pregrado.uoh.cl}

\author*[2]{\fnm{Gonzalo} \sur{Mu\~noz}}\email{gonzalo.m@uchile.cl}

\affil[1]{\orgdiv{Engineering School}, \orgname{Universidad de O'Higgins}, \orgaddress{\city{Rancagua}, \country{Chile}}}

\affil*[2]{\orgdiv{Industrial Engineering Department}, \orgname{Universidad de Chile}, \orgaddress{\city{Santiago}, \country{Chile}}}

\abstract{
The non-convex nature of trained neural networks has created significant obstacles in their incorporation into optimization models. 
In this context,
Anderson et al. (2020) provided a framework to obtain the convex hull of the graph of a piecewise linear convex activation function composed with an affine function; this effectively convexifies activations such as the ReLU together with the affine transformation that precedes it.
In this article, we contribute to this line of work by developing a recursive formula that yields a tight convexification for the composition of an activation with an affine function for a wide scope of activation functions, namely, convex or ``S-shaped". 
Our approach can be used to efficiently compute separating hyperplanes or determine that none exists in various settings, including non-polyhedral cases.
}


\maketitle

\section{Introduction}\label{sec:introduction}

In this article, we study the convexification of a set of the form 
\begin{equation}\label{mainset}
    S:=\{(x,y)\in D\times \mathbb{R}\,:\, y = \sigma(w^\top x + b)\}
\end{equation}
when $D\subseteq \mathbb{R}^n$ is a box domain and $\sigma:\mathbb{R}\to \mathbb{R}$ is a non-linear function.
This set appears naturally when embedding a trained neural network in an optimization model: here, $\sigma$ is the so-called \emph{activation function}, and $w^\top x + b$ is an affine transformation that is computed from one network layer to the next.

Embedding trained neural networks---which are highly expressive non-convex functions---in optimization models is a challenging and impactful topic, with applications in adversarial examples generation~\cite{cheng2017maximum,fischetti2018deep,khalil2018combinatorial}, verification and robustness certification~\cite{Tjeng2017Nov,li2022sok,liu2021algorithms,zhao2024bound}, 
network compression~\cite{serra2020lossless,serra2021scaling,elaraby2023oamip}, 
neural surrogate models~\cite{perakis2022optimizing,tong2023walk}, to mention a few.
Also see \cite{huchette2023deep}.
Well-established optimization solvers have recently incorporated modeling capabilities for this embedding (e.g. \cite{gurobi,turner2023pyscipopt}).
In this setting, convexifying a set of the form \eqref{mainset}, for each neuron, yields a convex relaxation.

One way of convexifying $S$ is to consider \(\{(x,y,z)\in D\times \mathbb{R} \times \mathbb{R}\,: \, y = \sigma(z),\, z = w^\top x + b\}\),
and to compute the convex hull of the graph of $\sigma$ in one dimension \cite{liu2021algorithms,schweidtmann2019deterministic,wilhelm2023convex}. 
The main limitation of this approach is that it considers each one-dimensional $\sigma$ separately, which can render weak relaxations \cite{salman2019convex}. 

One path to improve these relaxations was given in \cite{tjandraatmadja2020convex,anderson2020strong}. In these works, the authors presented a method for obtaining the convex hull of $S$ when $D$ is a polyhedron and $\sigma$ is piecewise linear convex. This includes the popular ReLU activation function $\sigma(z) = \max\{0,z\}$.
The convex hull, in this case, is polyhedral and may have exponentially many facets. One of the key contributions in \cite{tjandraatmadja2020convex,anderson2020strong} is that even if there are many facets, an efficient---i.e., polynomial time---separation exists. The authors also show extensive computational experiments advocating for the importance of tightly convexifying $S$.
We note that \cite{tawarmalani2013explicit} provides an alternative method to convexify $S$ for convex $\sigma$.

\paragraph{Contribution.}
In this paper, we follow the line of work of \cite{tjandraatmadja2020convex,anderson2020strong,tawarmalani2013explicit} and show that a tight and efficient convexification of $S$ is also achievable for a large class of $\sigma$ functions when $D$ is a box. 
Through the lens of function envelopes, we develop a recursive formula that yields separating hyperplanes for $\conv(S)$ using at most $n$ one-dimensional convexifications.
The $\sigma$ functions we handle are convex or ``S-shaped'' (defined below); while \cite{tawarmalani2013explicit} also produces a formula for the case of convex $\sigma$, to the best of our knowledge, no prior work extends beyond this case. Our recursive description offers, from a \emph{unified} point of view, a novel perspective on the convex case and significantly expands the reach of this type of approach to other activations.

\paragraph{Notation.}
We denote by $\mathbb{R}^n_{+}$ the set of vectors where \emph{all} their entries are \emph{strictly positive}.
We denote as $e_i\in \mathbb{R}^n$ the $i$-th canonical vector and $\mathbf{1}$ the vector of all ones. For $x\in \mathbb{R}^n$, we denote $x_{-i}$ the $(n-1)$-dimensional vector given by $x_{-i} = (x_1,\ldots, x_{i-1}, x_{i+1}, \ldots, x_n)$. For a function $f:[0,1]^n \to \mathbb{R}$ we denote $f_{-i}:[0,1]^{n-1}\to \mathbb{R}$ the function obtained by restricting $f$ to have the $i$-th input to $1$. To avoid ambiguity, we always write the argument of $f_{-i}$ as $x_{-i}$, thus 
\(f_{-i}(x_{-i}) = f(x_1,\ldots, x_{i-1},1,x_{i+1},\ldots, x_n).\)

For a set $S\subseteq \mathbb{R}^n$, we denote $\inte(S)$, $\cl(S)$ and $\conv(S)$ its interior, closure, and convex hull, respectively.
For a function $f$, we denote $\conc(f,S)$ its concave envelope over $S$. This is the smallest concave function that upper bounds $f$ on the domain $S$. Similarly, we denote $\conv(f,S)$ its convex envelope over $S$.
Finally, and slightly abusing notation, if $S_1\subseteq S_2$, we say $\conc(f,S_1) = \conc(f,S_2)$ if the concave envelopes coincide on the smallest set, i.e., on $S_1$.

\section{Literature review} \label{sec:litreview}

\paragraph{Approximation of trained neural networks.}
Multiple articles have leveraged the convex hull of $\{(z,y)\, :\, y = \sigma(z)\}$  \cite{schweidtmann2019deterministic,wilhelm2023convex,ehlers2017formal}. However, this can yield weak relaxations \cite{salman2019convex}. 
One research line has improved this by convexifying multiple activations of a single layer simultaneously. 
See, e.g., \cite{singh2019beyond,roth2021primer,muller2022prima,ferrari2022complete,zhang2022general,tang2023boosting,ma2024relu}.
Roughly speaking, they consider groups of constraints of the form $y = \sigma(z)$ and, starting from polyhedral constraints on the $z$ variables, they construct constraints that involve groups of $y$ variables. 
In these approaches, obtaining convex hulls is usually intractable, and thus, the authors rely on tractable polyhedra for the $z$ variables and group a small number of activations. 

A different approach was taken in \cite{tjandraatmadja2020convex,anderson2020strong}, where the authors consider the set $S$ defined in \eqref{mainset} when $D$ is a polytope and $\sigma$ is piecewise linear convex. Specifically, they construct an \emph{ideal formulation} for $S$, which yields the convex hull of $S$.
Related strategies have been proposed in \cite{kronqvist2021steps,tsay2021partition,kronqvist2022psplit}.

Other approaches that handle non-convexities of a neural network include using convex \emph{restrictions} \cite{bastani2016measuring}, piecewise-linear approximations \cite{sildir2022mixed}, piecewise-linear \emph{relaxations}  \cite{benussi2022individual}, Lagrangian relaxations \cite{dvijotham2018dual}, semidefinite relaxations \cite{raghunathan2018semidefinite}, and convexifications of the softmax \cite{wei2023convex}. Also see \cite{salman2019convex}.

\paragraph{Convex and concave envelopes.}
In \cite{tawarmalani2013explicit}, the authors study envelopes of various functions over subsets of a hyper-rectangle. In particular, they 
provide a closed-form formula for $\conc(f,D)$ when $f$ is convex and $D$ is a box; this can handle the case of convex $\sigma$. 
Beyond the convex case, we briefly discuss articles that consider non-polyhedral envelopes in arbitrary dimensions (which is the case for S-shaped $\sigma$).
In \cite{tawarmalani2001semidefinite}, the authors obtain the convex envelope for functions that are concave on one variable and convex on the rest. 
In \cite{jach2008convex}, the authors show how to evaluate $\conc(f,R)$ when $f$ is an $(n-1)$-convex function and $R$ is a rectangular domain. 
In~\cite{khajavirad2012convex,khajavirad2013convex}, the authors formulate the convex envelope of a lower semi-continuous function over a compact set as a convex optimization problem. 
In \cite{barrera2022convex}, the authors compute the envelope of ray-concave functions.

To the best of our knowledge, most of these approaches do not apply to functions of the type $\sigma(w^\top x + b)$, when $\sigma$ is neither concave nor convex.
The only exception is the formulation considered in \cite{khajavirad2012convex,khajavirad2013convex}, which is very general. However, it is unclear to us how to exploit this formulation in order to obtain a simple expression of the envelope.

\section{Main result and discussion}\label{sec:mainresults}
\subsection{Definitions and reductions}
We aim to describe $\conv(S)$, where $S$ is defined in \eqref{mainset} with $D$ a bounded box domain and $\sigma$ continuous. 
Let $f(x) = \sigma(w^\top x + b)$.
In this setting,
\begin{equation}\label{eq:hulltoenvelopes}
    \conv\left(S\right)\, =\, \{(x,y)\in D\times \mathbb{R}\,:\, \conv(f,D) \leq y \leq \conc(f,D)\},
\end{equation}
i.e., computing the convex and concave envelopes of $f$ suffices.

\begin{definition}
    We say a continuous function $\sigma:\mathbb{R}\to \mathbb{R}$ is S-shaped, if there exists $\tilde{z}$ such that $\sigma:(-\infty,\tilde{z}]\to \mathbb{R}$ is convex and $\sigma:[\tilde{z},\infty)\to \mathbb{R}$ is concave.
\end{definition}

The \emph{differentiable} S-shaped functions are commonly known as sigmoid functions.
In our case, we do not assume differentiability.
We also note that in deep learning, the term \emph{sigmoid} is typically reserved for the logistic function $1/(1+e^{-x})$.
To unify our discussion, we define the following property.

\begin{definition}\label{def:stfe}
We say a continuous function $\sigma:\mathbb{R} \to \mathbb{R}$ satisfies the \emph{secant-then-function envelope} (STFE) property if  
\begin{equation} \label{eq:stfe}
\forall [\ell,u]\subseteq \mathbb{R},\, \exists \hat{z}\in [\ell, u],\, \conc(\sigma,[\ell,u])(z) = \left\{ \begin{array}{ll}
     \sigma(\ell) + \frac{\sigma(\hat{z}) - \sigma(\ell)}{\hat{z} - \ell} (z-\ell) & \text{if } z\in [\ell, \hat{z}) \\
     \sigma(z)& \text{if } z\in [ \hat{z}, u]
\end{array}\right.
\end{equation}

For a fixed $[\ell,u]$, we call the \emph{smallest} point $\hat z$ satisfying \eqref{eq:stfe} the \emph{tie point} of $\sigma$ over $[\ell,u]$.
\end{definition}

While multiple $\hat{z}$ may satisfy \eqref{eq:stfe}, the tie point is uniquely defined.
It is not hard to see that if $\sigma$ is either convex, concave, or S-shaped, then the STFE property is satisfied. In the convex and concave case, it is direct. For the differentiable S-shaped case, see \cite{wilhelm2023convex,maranas1995finding,liberti2003convex}. We have not found a proof for the non-differentiable case, but using results from \cite{tawarmalani2002convex}, it follows.\\

To state our main technical result, we simplify the setting as follows:
\begin{enumerate}
    \item We can assume $w_i > 0$, for all $i=1,\ldots, n$, by reducing the dimension when $w_i = 0$ and changing variables $x_i \to -x_i$ when $w_i < 0$.
    \item Additionally, we can rescale and shift the variables to obtain $x\in [0,1]^n$. This changes $w$ and $b$ but does not change the signs of any $w_i$.
\end{enumerate}
It is easy to see that computing envelopes in this modified setting is equivalent. 

\subsection{Main result: concave envelopes}
Our main result is a recursive formula for the \emph{concave} envelope of $f$ defined above, after the previous reductions, when $\sigma$ satisfies the STFE property.
For the construction, we define the following subdivision of $[0,1]^n$.

\begin{definition}\label{def:regions}
 Given $w\in\mathbb{R}^n_{+} $, $b\in \mathbb{R}$, and $\hat{z} \in \mathbb{R}$, we define
     \begin{align*}
    R_f &:= \{ x\in [0,1]^n\, :\, w^\top x + b \geq \hat{z} \} \\
    R_l &:= \{ x\in [0,1]^n\, :\, w^\top x + b < \hat{z},\, w^\top x + b\|x\|_\infty \geq \hat{z}\|x\|_\infty\} \\
    R_i &:= \{ x\in [0,1]^n\, :\, w^\top x + b\|x\|_\infty < \hat{z}\|x\|_\infty, \\
    &\qquad \qquad \qquad \quad  x_i > x_j\quad  \forall j <i,\, x_i \geq x_j\quad  \forall j \geq i\} && i=1,\ldots, n
\end{align*}
\end{definition}
Some regions may be empty or low-dimensional, and they partition $[0,1]^n$. The sub-indices in the names of the regions make reference to the fact that the envelope will be equal to $f$ in $R_f$, linear in $R_l$, and that the $i$-th component of $x$ is the largest in $R_i$. 
Our main result is the following.
\begin{theorem}\label{thm:main}
Consider $w\in\mathbb{R}^n_{+} $, $b\in \mathbb{R}$ and let $f:[0,1]^n \to \mathbb{R}$ be a function of the form
\(f(x) = \sigma(w^\top x + b)\) 
where $\sigma$ satisfies the STFE property. Then,
\[
\conc(f,[0,1]^n)(x) = \left\{ \begin{array}{ll}
     f(x)& \text{if } x\in R_f \\
     \sigma(b) + \frac{\sigma(\hat{z}) - \sigma(b)}{\hat{z} - b} (w^\top x) & \text{if } x\in R_l\\
     \sigma(b) + x_i \conc(f_{-i} - \sigma(b),[0,1]^{n-1})\left(\frac{x_{-i}}{x_i}\right)  & \text{if } x\in R_i\\
\end{array}\right.
\]
where $R_f,R_l$ and $R_i$ are defined using Definition \ref{def:regions} with $\hat{z}$ as the tie point of $\sigma$ in $[b,w^\top \mathbf{1} + b]$.
\end{theorem}

Note that in $R_i$ the envelope is defined using a perspective function, which has played an important role in convexification techniques \cite{gunluk2011perspective}.
Theorem \ref{thm:main} states that evaluating $\conc(f,[0,1]^n)(x)$ for a given $x$ is either a direct function evaluation in $R_l$ or $R_f$, or an envelope evaluation in one lower dimension. Since we require a tie point for constructing the regions, in the worst case, we need to compute a one-dimensional envelope of $\sigma$ for $n$ different intervals.
Theorem \ref{thm:main} is also useful for computing gradients of the envelope, almost everywhere, using the chain rule.
This can be useful, e.g., when strengthening an initial relaxation of $\conv(S)$ via separating hyperplanes.

\begin{example}\label{ex:main}
    Consider the sigmoid activation
    \(\sigma(x) = 1/(1+e^{x}),\)
    $w=(10,5)$, and $b=-10$.
    We display the function and envelope in Figure \ref{fig:functionandslices-env}. 
    Note that the concave envelopes that are computed recursively will not be affine in general.  
    Additionally, the tie point used to compute the envelope in both facets is not on the line $w^\top x + b = \hat{z}$, where $\hat z$ is the tie point of $\sigma$ in $[b, w^\top \mathbf{1} + b]$. This is because, in each facet, the argument of $\sigma$ has a different range. However, in both cases, the tie point on the slice is ``to the left'' of the points given by $w^\top x + b = \hat{z}$. This is not a coincidence; we prove below that this always happens and, moreover, we leverage it in our arguments.
\end{example}

\begin{figure}[t]
    \centering
    \begin{subfigure}{0.5\textwidth}
        \centering
                \begin{overpic}[abs,unit=1mm,width=0.75\textwidth]{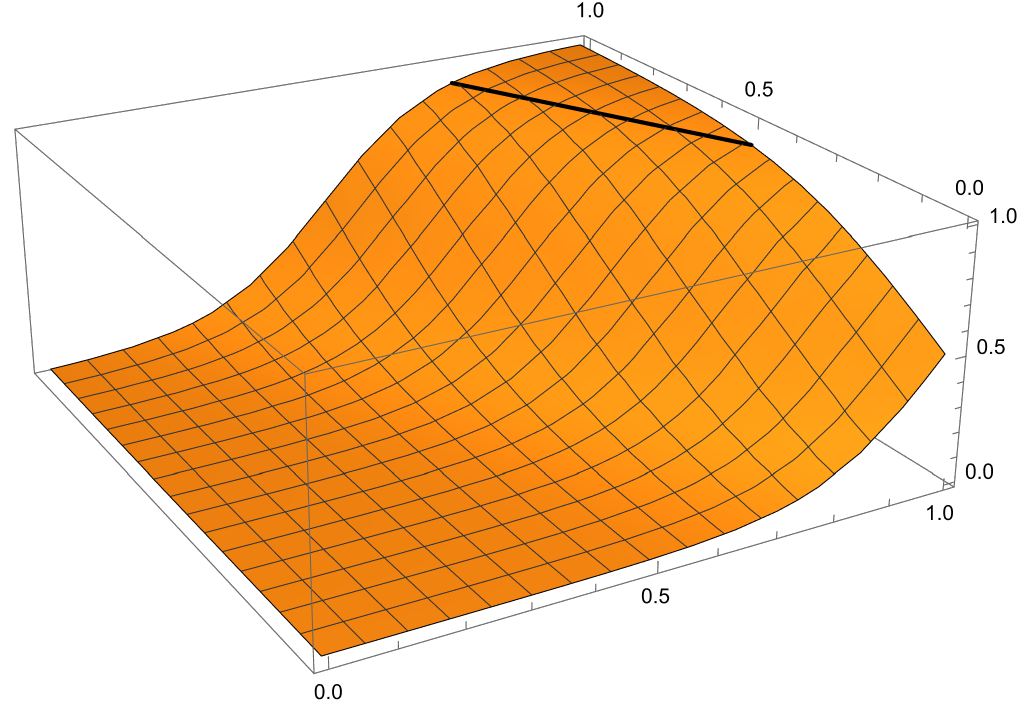}
\put(27,3){\color{black}$x_1$}
\put(3,9){\color{black}$x_2$}
\end{overpic}
\begin{overpic}[abs,unit=1mm,width=0.73\textwidth]{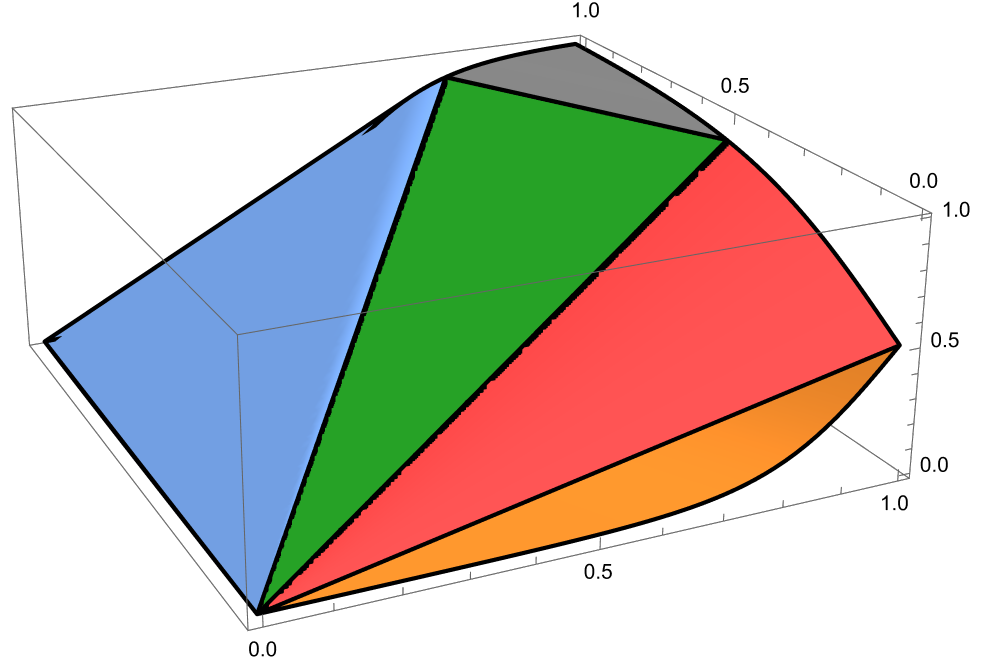}
\put(27,2){\color{black}$x_1$}
\put(1,9){\color{black}$x_2$}
\end{overpic}
        \caption{Top: plot of $f(x)$ (orange) and line $w^\top x + b = \hat{z}$ (black). Bottom: the concave envelope divided into the regions $R_f$ (gray), $R_l$ (green), $R_1$ (red), and $R_2$ (light blue).}
    \end{subfigure}%
    \hspace{.2cm}
    \begin{subfigure}{0.4\textwidth}
        \centering
        \includegraphics[width=0.85\textwidth]{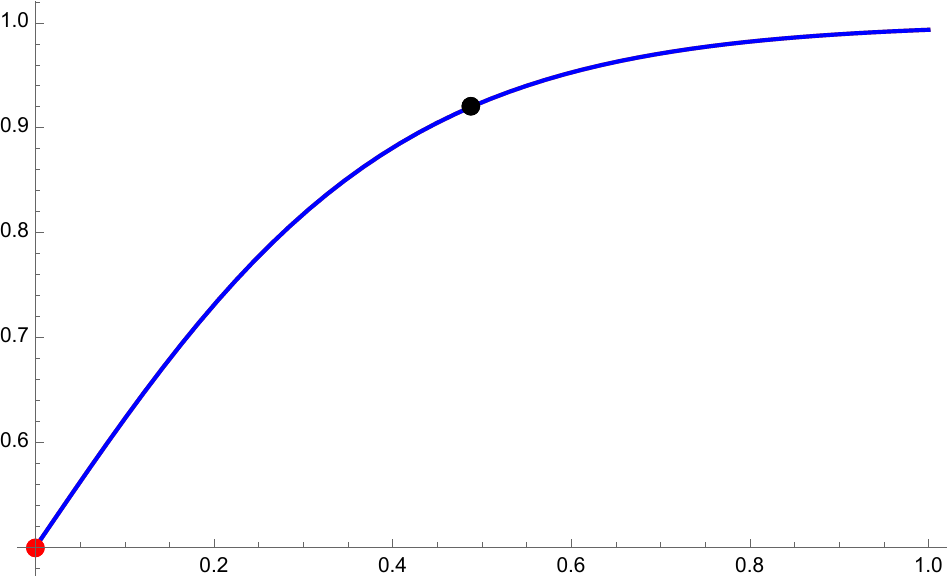}
        \includegraphics[width=0.85\textwidth]{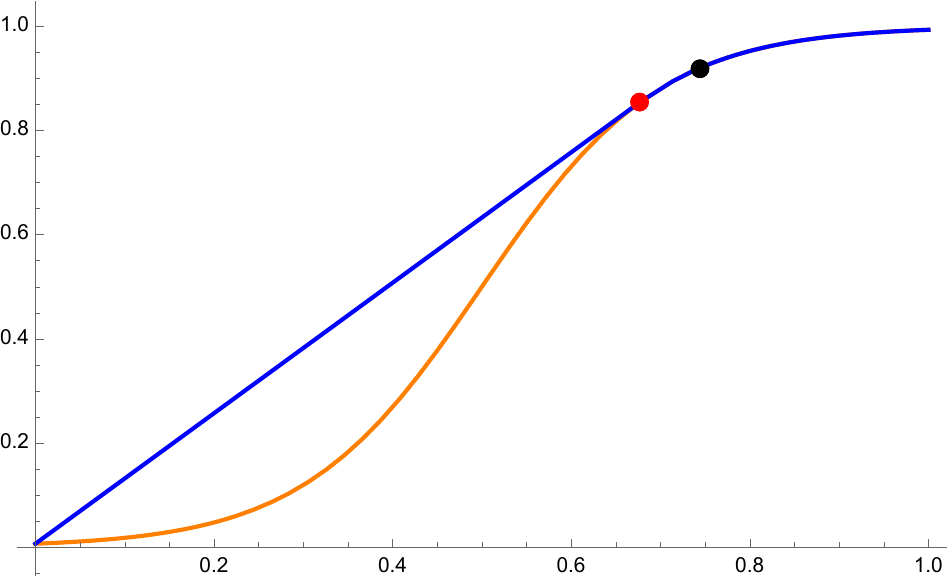}
        \caption{Slices of $f(x)$ (orange) and its envelope (blue) when $x_1 = 1$ (top) and $x_2=1$ (bottom). The black points are the intersection of the line $w^\top x + b = \hat{z}$ with each slice, and the red points indicate the tie point on the slice.}
    \end{subfigure}
    \caption{Plots of $f(x) = \sigma(w^\top x + b)$ as defined in Example \ref{ex:main}, its concave envelope computed with Theorem \ref{thm:main}, and the slices given by setting the variables to 1.} \label{fig:functionandslices-env}
\end{figure}

\paragraph{Reach of Theorem \ref{thm:main}.}
\label{sec:reach}

While the STFE property may seem restrictive, a large portion of the most commonly used activation functions in deep learning satisfy it. In particular, all convex and S-shaped activations. In Table \ref{table:activations} of Appendix \ref{app:list}, we list common activations of these types.

In addition, Theorem \ref{thm:main} apparently only yields \emph{one} of the envelopes needed in \eqref{eq:hulltoenvelopes}.
However, convex envelopes can also be obtained. When $\sigma$ is convex, the convex envelope is $\sigma$ itself. When $\sigma$ is S-shaped, $\sigma'(z) = -\sigma(-z)$ is also S-shaped, and computing the \emph{concave} envelope of $\sigma'(w^\top y - b)$ yields the \emph{convex} envelope of $\sigma(w^\top x + b)$ using the change of variables $y=-x$.

\paragraph{Connection to non-recursive formula in the convex case.}
Fix $x\in[0,1]^n$ and let $\pi$ be a permutation such that $x_{\pi(1)} \geq x_{\pi(2)} \geq \cdots \geq x_{\pi(n)}$. In \cite[Corollary 3.14]{tawarmalani2013explicit}, the authors show that if $\sigma$ is convex, then
\begin{equation}\label{eq:lovaszformula}
\conc(f,[0,1]^n)(x) = \sum_{i=1}^{n} 
   \Big(
      f\Big( \sum_{j=1}^{i} e_{\pi(j)} \Big)
      - 
      f\Big( \sum_{j=1}^{i-1} e_{\pi(j)} \Big)
   \Big)
   x_{\pi(i)}
   + f(0).
   \end{equation}
Let us recover \eqref{eq:lovaszformula} via Theorem \ref{thm:main}.
For brevity, we avoid edge cases by assuming $\sigma$ is not affine. Then, the tie point of $\sigma$ is always $w^\top \mathbf{1} + b$, and the only regions with non-empty interior are $R_i$. Also note that $\pi$ sorts the elements of $x/x_i$ as well. We proceed by induction on $n$: the base case is direct, and for the inductive step, take $x\in R_k$, then $\pi(1) = k $ and thus

\begin{align}
& \conc(f,[0,1]^n)(x)
= \, \sigma(b) + x_k \conc(f_{-k} - \sigma(b),[0,1]^{n-1})\Big(\frac{x_{-k}}{x_k}\Big) \nonumber \\
 = &\,  x_k\Big[ \sum_{i=2}^{n} \Big(
      f\Big( \sum_{j=1}^{i} e_{\pi(j)} \Big)
      - 
      f\Big( \sum_{j=1}^{i-1} e_{\pi(j)} \Big)
   \Big)
   \frac{x_{\pi(i)}}{x_k}  + f(e_{\pi(1)}) - \sigma(b) \Big] + \sigma(b)\label{eq:two}\\
= &\,  \sum_{i=2}^{n} \Big(
      f\Big( \sum_{j=1}^{i} e_{\pi(j)} \Big)
      - 
      f\Big( \sum_{j=1}^{i-1} e_{\pi(j)} \Big)
   \Big)
   x_{\pi(i)}  + x_{k}(f(e_{\pi(1)}) - f(0) ) + f(0)\label{eq:last}
\end{align}

The sum starting from $i=2$ in \eqref{eq:two} comes from applying \eqref{eq:lovaszformula} to $f_{-k}$, which is obtained from $f$ by fixing the $k$-th component to 1. Formula \eqref{eq:lovaszformula} is obtained from \eqref{eq:last} by noting that $x_{k}(f(e_{\pi(1)}) - f(0) )$ is the $i=1$ term of the sum. \\

The reader may wonder if a non-recursive formula akin to \eqref{eq:lovaszformula} is possible in the S-shaped case.
While one can certainly expand the recursion in Theorem \ref{thm:main}, there are two key issues that make it difficult to derive a clean, simple closed-form expression.
Firstly, unlike the convex case, evaluating the envelope in the S-shaped case requires more than evaluating $f$ at the vertices of the domain. And secondly, the tie points (used for computing the regions $R_f, R_l$, and $R_i$) may change when restricting to lower-dimensional faces.
The recursion in Theorem \ref{thm:main} is the simplest expression we have found yet, since it incorporates these two issues into the recursion itself.

\section{Preliminary lemmas} \label{sec:proofs}

We begin by listing a number of simple properties for a $\sigma$ function that satisfies STFE.
Since they are rather intuitive, we relegate their proofs to Appendix \ref{app:oned}. We note that Case 2 of the following lemma, in particular, is not direct.

\begin{lemma}\label{lemma:onedprops}
Let $\sigma:\mathbb{R}\to \mathbb{R}$ be a function satisfying STFE, with a tie point $\hat{z}$ on an interval $[L,U]$. Then the following properties hold:
\begin{enumerate}
    \item \textbf{(Decreasing upper bound, case 1)} For every $\tilde{U}\in [\hat{z}, U]$,  
    \[\conc(\sigma,[L, \tilde{U}])(z) = \conc(\sigma,[L, U])(z)\quad \forall z\in [L, \tilde{U}]\]
    and $\hat{z} $ is the tie point of $\sigma$ in the interval $[L, \tilde{U}]$ as well.
    \item \textbf{(Decreasing upper bound, case 2)} For every $\tilde{U}\in [L, \hat{z}]$,  
    \[\conc(\sigma,[L, \tilde{U}])(z) = \sigma(L) + \frac{\sigma(\tilde{U}) - \sigma(L)}{\tilde{U} - L} (z-L).\]
    \item \textbf{(Increasing lower bound)} For $\tilde{L} \in [L, \hat{z}]$ let $\tilde{z}$ be the tie point of $\sigma$ in $[\tilde{L}, U]$. Then $\tilde{z} \leq \hat z$.
\end{enumerate}
\end{lemma}

We also state a result from \cite{tawarmalani2002convex} (rephrased), which we will repeatedly use.

\begin{lemma}{\cite[Corollary 2]{tawarmalani2002convex}}\label{cor:envinfaces}
    Let $D$ be a convex set and $D'$ be a non-empty face of $D$. Assume $\phi(x)$ is the concave envelope of $f(x)$ over $D$. Then, the restriction of $\phi(x)$ to $D'$ is the concave envelope of $f(x)$ over $D'$.
\end{lemma}

We continue with a ``baseline'' concave overestimator for $f$.

\begin{lemma}\label{lemma:mccormick}
Consider $w,b,\sigma,f$ as in Theorem \ref{thm:main}. Define $h:[0,1]^n \to \mathbb{R}$ as
\[
h(x) = \left\{ \begin{array}{ll}
f(x)&\text{if } x\in R_f \\ 
     \sigma(b) + \frac{\sigma(\hat{z}) - \sigma(b)}{\hat{z} - b} (w^\top x) & \text{otherwise} \\
      
\end{array}\right.
\]
where $R_f$ is defined using Definition \ref{def:regions} with $\hat{z}$ as the tie point of $\sigma$ in $[b,w^\top \mathbf{1} + b]$. Then $h$ is a concave overestimator of $f$.
\end{lemma}
\begin{proof}
    This proof follows directly since $h(x) = \conc(\sigma,[b,w^\top \mathbf{1} + b])(w^\top x + b)$. 
\end{proof}

\begin{lemma}\label{lemma:htight}
    Consider $h:[0,1]^n \to \mathbb{R}$ as in Lemma \ref{lemma:mccormick} and the regions defined in Definition \ref{def:regions}. For any $x \in R_f \cup R_l$, $h(x) = \conc(f,[0,1]^n)(x)$.
\end{lemma}
\begin{proof}
    If $x\in R_f$, the result is trivial. For $x\in R_l$, it follows since it $h(x)$ is determined by the secant between two points where $h$ is equal to $f$.
\end{proof}
We conclude with three tie-point-related properties.
\begin{lemma}\label{lemma:biggerthantie}
Let $\hat{z}$ be the tie point of $\sigma$ in $[b,w^\top \mathbf{1} + b]$. For every $z \geq \hat{z}$,
\[\conc(f,[0,1]^{n}) (x) = \conc(f,[0,1]^{n} \cap \{x\,:\, w^\top x + b \leq z\} ) (x).\]
\end{lemma}
\begin{proof}
    Using Lemma \ref{lemma:htight}, it must be that
    \(f(x) = \conc(f,[0,1]^{n} \cap \{x\,:\, w^\top x + b \leq z\} )(x) \)
    for every $x$ such that $w^\top x + b\in [\hat{z}, z]$. This implies the result.
\end{proof}

\begin{lemma}\label{lemma:eiinRi}
Let $\hat{z}$ be the tie point of $\sigma$ in $[b,w^\top \mathbf{1} + b]$, and define the region $R_i$ according to Definition \ref{def:regions}.
If $R_i\neq \emptyset$, then $e_i \in R_i$, where $e_i$ is the $i$-th canonical vector. In particular, this implies $w_i + b < \hat z$.
\end{lemma}
\begin{proof}
Take $x\in R_i$. Clearly $x\neq 0$.  Then we can take $\tilde{x} = x/\|x\|_\infty$ which satisfies
\(w^\top \tilde{x} + b < \hat{z}.\)
Additionally, $\tilde{x}_i = 1$, thus
\(\hat{z} - b > w^\top \tilde{x} \geq w_i = w^\top e_i.\)
\end{proof}

\begin{lemma}\label{lemma:concinRi}
    Let $\hat{z}$ be the tie point of $\sigma$ in $[b,w^\top \mathbf{1} + b]$, and suppose $R_i\neq \emptyset$. Then,
    \(\conc(f_{-i},[0,1]^{n-1})  = \conc(f_{-i},[0,1]^{n-1} \cap \{x_{-i}\,:\, w_{-i}^\top x_{-i} + w_i + b \leq \hat z\} ). \)
\end{lemma}

Before the proof, let us provide some intuition. Consider Figure \ref{fig:functionandslices-env}(b): Lemma \ref{lemma:concinRi} states that if we restrict the function until the black point, the envelope is the same as in the complete interval. In other words, the tie point on the slice is to the left of the black point. This is what the proof formalizes.

\begin{proof}
    As
    \(f_{-i}(x_{-i}) = \sigma (w_{-i}^\top x_{-i} + w_i + b),\)
    the domain of $\sigma$ in this case is $[w_i + b, w^\top \mathbf{1} + b]$, while in $f$ it is $[b, w^\top \mathbf{1} + b]$. 
    Let $\hat{z}^{-i}$ be the tie point of $\sigma$ in $[w_i + b, w^\top \mathbf{1} + b]$. By Lemma \ref{lemma:eiinRi} we have $w_i + b < \hat{z}$, so we can apply Lemma \ref{lemma:onedprops} (case 3) to obtain
    \(\hat{z} \geq \hat{z}^{-i}.\)
    To conclude, we apply Lemma \ref{lemma:biggerthantie} to $f_{-i}$. 
\end{proof}

\section{Concave envelope proof} \label{sec:mainproof}

To alleviate notation in the proofs, we assume in this section $f(0) = \sigma(b) = 0$. This can be achieved by considering $\sigma'(z) = \sigma(z) - \sigma(b)$ instead of $\sigma$.

\begin{definition}\label{def:g}
    Consider $w,b,\sigma,f$ as in Theorem \ref{thm:main} with $\sigma(b)= 0$. Define
\[
g(x) = \left\{ \begin{array}{ll}
     f(x)& \text{if } x\in R_f \\
     \frac{\sigma(\hat{z})}{\hat{z} - b} (w^\top x) & \text{if } x\in R_l\\
      x_i \conc(f_{-i},[0,1]^{n-1})\left(\frac{x_{-i}}{x_i}\right)  & \text{if } x\in R_i,\, i=1,\ldots n \\
\end{array}\right.
\]
where $R_f,R_l$ and $R_i$ are defined using Definition \ref{def:regions} with $\hat{z}$ as the tie point of $\sigma$ in $[b,w^\top \mathbf{1} + b]$.
\end{definition}

 We show that $g$ is precisely the concave envelope of $f$ in $[0,1]^n$ in this case; Theorem \ref{thm:main} is then obtained simply by shifting using $\sigma(b)$. Our strategy will be to use induction on the dimension $n$.

\subsection{General properties of the envelope candidate}

\begin{lemma}\label{lemma:gcontinuous}
    Consider $g$ as in Definition \ref{def:g}.
Then, $g$ is continuous on $[0,1]^n$.
\end{lemma}
\begin{proof}
    Firstly, the function $g$ is equal to $h$ in $R_f\cup R_l$, which is continuous.
In each $R_i$ individually, the function is concave (thus continuous) since it is the perspective of a concave function \cite{gunluk2011perspective}. In $x=0$, all definitions evaluate to 0. 

If we take $x\in \cl(R_f) \cap \cl(R_i)$, with $x\neq 0$, we have
    \(\hat{z} \geq w^\top x/\|x\|_\infty + b  \geq w^\top x + b \geq \hat{z}. \)
    Since $w\in \mathbb{R}^n_{+}$ and $x\neq 0$, $w^\top x \neq 0$ and thus the previous inequalities imply $\|x\|_\infty = 1$. This implies $x_i = 1$ and 
    \[x_i \conc(f_{-i},[0,1]^{n-1})\left(x_{-i}/x_i\right) = f_{-i}(x_{-i}) =  f(x).\]
    where the first equality follows from the same argument in the proof of Lemma \ref{lemma:concinRi} (i.e., the tie point used in $f_{-i}$ is smaller than $\hat z$). 

    If we take $x\in \cl(R_l) \cap \cl(R_i)$, with $x\neq 0$, we have
    \( w^\top x/\|x\|_\infty + b  = \hat{z} \)
    and thus $x/\|x\|_\infty \in \cl(R_l) \cap \cl(R_i) $. 
    Therefore, the definition of $g$ on $R_l$ evaluated on $x/\|x\|_\infty$ is
    \[\frac{\sigma(\hat{z})}{\hat{z} - b} (w^\top x/\|x\|_\infty)  = \sigma(\hat{z}),\]
    and the definition of $g$ on $R_i$ evaluated on $x/\|x\|_\infty$ is
    \[
    \conc(f_{-i},[0,1]^{n-1})\left(x_{-i}/\|x\|_\infty \right) = f_{-i}\left(x_{-i}/\|x\|_\infty \right) = f(x/\|x\|_\infty),
    \]
    where the first equality follows from the same argument in the proof of Lemma \ref{lemma:concinRi}.
    Note that $f(x/\|x\|_\infty) = \sigma(w^\top x/\|x\|_\infty + b) = \sigma(\hat z)$ and thus 
    \[\frac{\sigma(\hat{z})}{\hat{z} - b} (w^\top x/\|x\|_\infty)  = \conc(f_{-i},[0,1]^{n-1})\left(x_{-i}/\|x\|_\infty \right).\]
    We conclude this case by noting that both function definitions on $R_l$ and $R_i$ are linear on the rays, and thus, we can rescale the arguments.

    The last case to consider is $x\in \cl(R_i) \cap \cl(R_j)$ for $i\neq j$. In this case, we must have $x_i = x_j = \|x\|_\infty$, 
    and thus the $i$ and $j$ components of $x/\|x\|_\infty$ are 1. In what follows we use Lemma \ref{cor:envinfaces} twice:
    \begin{align*}
        \conc(f_{-i},[0,1]^{n-1})\left(x_{-i}/\|x\|_\infty \right) &= \conc((f_{-i})_{-j},[0,1]^{n-2})\left((x_{-i})_{-j}/\|x\|_\infty \right)\\
        &= \conc((f_{-j})_{-i},[0,1]^{n-2})\left((x_{-j})_{-i}/\|x\|_\infty \right)\\
        &= \conc(f_{-j},[0,1]^{n-1})\left(x_{-j}/\|x\|_\infty \right).
    \end{align*}
    Since $x_i=x_j$, we conclude the result.
\end{proof}
\begin{lemma}\label{lemma:goverest}
Consider $g$ as in Definition \ref{def:g}.
Then, $g$ is an overestimator of $f$.
\end{lemma}
\begin{proof}
    In $R_l\cup R_f$ it follows since $g(x) = h(x)$ in these regions. 
    Let us now consider $x\in R_i$ fixed and let $\lambda \geq 0$. Since in $R_i$ the function $g$ is defined as the perspective of $\conc(f_{-i},[0,1]^{n-1})$, we have 
    \(g(\lambda x) = \lambda g(x).\)
    Moreover, $g(0) = 0 = f(0)$ and $g(x/\|x\|_\infty) = \conc(f_{-i},[0,1]^{n-1})\left(x_{-i}/\|x\|_\infty\right) \geq f(x/\|x\|_\infty)$, therefore, $\alpha(\lambda) := g(\lambda x/\|x\|_\infty)$ defines an overestimator of the secant $\beta(\lambda) = \lambda f(x/\|x\|_\infty)$ for $\lambda\in[0, 1]$.
    We conclude by noting that the one-dimensional function 
    \[\gamma(\lambda) := f(\lambda x/\|x\|_\infty) = \sigma(\lambda w^\top x/\|x\|_\infty + b)\]
    satisfies the STFE property since $w^\top x/\|x\|_\infty > 0$. 
    The argument of $\sigma$ ranges in $[b,w^\top x/ \|x\|_\infty + b]$, and $w^\top x /\|x\|_\infty + b < \hat{z}$ since $x\in R_i$. Thus, by Lemma \ref{lemma:onedprops} (case 2), $\conc(\gamma,[0,1])(\lambda) = \lambda \gamma(1)$. In particular,
    \begin{align*}
        \underbrace{f(\lambda x / \|x\|_\infty)}_{\gamma(\lambda)} &\leq \underbrace{\lambda f(x/\|x\|_\infty)}_{\lambda \gamma(1) = \beta(\lambda)} \leq \underbrace{g(\lambda x/\|x\|_\infty)}_{\alpha(\lambda)}
    \end{align*}
    for every $\lambda\in [0,1]$.
    The results follows taking $\lambda = \|x\|_\infty \leq 1$.
\end{proof}

\vskip .5cm
\begin{lemma}\label{lemma:envelopeinregion}
    The function $g$ defined as in Definition \ref{def:g} is concave in each region. Moreover, $g$ is the concave envelope of each region separately.
\end{lemma}
\begin{proof}
For regions $R_l$ and $R_f$, the result is direct.
    For $R_i$,
        \begin{align*}
        g(x) 
        =&\, x_i \conc(f, [0,1]^{n} \cap R_i \cap \{x_i = 1\}) (x/x_i) && \text{(Lemma \ref{lemma:concinRi})}\\
        \leq&\, x_i \conc(f, R_i) (x/x_i) + (1-x_i) \conc(f, R_i) (0) && \text{($\conc(f, R_i) (0) \geq  0$)}\\
        \leq &\, \conc(f, R_i) (x).
        \end{align*}
        Since $g$ is also a concave overestimator (Lemma \ref{lemma:goverest}), it must be the concave envelope.
\end{proof}
Finally, we state a lemma that follows directly from the definition of $g$.

\begin{lemma} \label{lemma:linearorf}
For every $x\in [0,1]^n$, either (i) $g(x) = f(x)$, or
(ii) $g(x)$ is linear on the segment $[0,x]$.
\end{lemma}

With these results, we are only missing the global concavity of $g$.
As mentioned earlier, we prove this by induction on the dimension $n$. The base case of $n=1$ is easy, and we omit it for brevity; it essentially relies on the fact that $R_1 =\emptyset$ when $n=1$.
In the next subsection, we focus on the inductive step.

\subsection{Global concavity: Inductive step}

Let us assume that for dimension $\leq n-1$, the concave envelope of a function $\sigma(w^\top x + b)$ can be expressed using $g$ as in Definition \ref{def:g}.
To prove the concavity of $g$ in $n$ dimensions, it suffices to consider $x\in (0,1)^n$ and show that for any $y$ in a dense subset of $[0,1]^n$ (to be determined), 
\begin{equation}\label{eq:gradientineq}
    f(x) \leq g(y) + \nabla g(y)^{\top}(x-y).
\end{equation}

Why does this suffice? For each fixed $y$, the right-hand side of \eqref{eq:gradientineq} defines a linear overestimator of $f$, and since $g$ is the concave envelope in each separate region (Lemma \ref{lemma:envelopeinregion}), we get $g(x) \leq g(y) + \nabla g(y)^{\top}(x-y)$ for every $x\in (0,1)^n$ and $y\in D'$, where $D'$ is a dense subset of $[0,1]^n$.

Next, we take $x,z\in (0,1)^n$ arbitrary and let $y := \lambda x + (1-\lambda) z$ for $\lambda \in (0,1)$. By density of $D'$, we can take a sequence $\{y_k\}_{k=1}^\infty \subseteq D'$ such that $y_k \to y$. With $x_k := x+(y_k -y)$ and $z_k := z + (y_k - y)$ we get $y_k = \lambda x_k + (1-\lambda) z_k$. Using $y_k \in D'$ and the discussion of the previous paragraph, 
\[g(x_k) \leq g(y_k) + \nabla g(y_k)^{\top}(x_k-y_k) \quad \land \quad g(z_k) \leq g(y_k) + \nabla g(y_k)^{\top}(z_k-y_k).\]

Combining these inequalities with multipliers $\lambda$ and $(1-\lambda)$, we obtain $\lambda g(x_k) + (1-\lambda) g(z_k) \leq g(y_k) = g(\lambda x_k + (1-\lambda) z_k)$. Taking limits and using that $g$ is continuous (Lemma \ref{lemma:gcontinuous}), we conclude the global concavity of $g$.\\

Henceforth, we focus on \eqref{eq:gradientineq}; we begin by reducing some cases. If $x,y$ are in the same region, it follows by concavity of $g$ in each region (Lemma \ref{lemma:envelopeinregion}). If $y\in R_l \cup R_f$, \eqref{eq:gradientineq} follows since $g(y) = h(y)$ and $h$ is a concave overestimator.
Therefore, we focus on $y\in R_i$ for some $R_i$: moreover, we can consider $y\in \inte(R_i)$ by density.
The last case we can discard is when $x\in R_f$: by the proof of Lemma \ref{lemma:gcontinuous}, we see that $\cl(R_i)\cap \cl(R_f)$ is on the boundary of $[0,1]^n$.
Thus, it suffices to show \eqref{eq:gradientineq} for $y$ in a dense subset of $\inte(R_i)$ and $x\not\in (R_i\cup R_f)$. 

Before showing the proof, let us describe the high-level strategy. Consider Figure \ref{fig:proofidea}; here, we display two possible ``locations'' for $y$ within $R_i$, which will become clearer next. We remark that these figures follow the same parameters as in Example \ref{ex:main}, and thus the reader can also refer to the previous figures.
We first find a direction $\tilde w$ that is orthogonal to $w$ such that $x+\alpha \tilde w \in R_i$ for some $\alpha$. This orthogonality ensures that the function $f$ stays constant. Then, we can use the concavity of $g$ over $R_i$ to overestimate $f(x+\alpha \tilde w)$ with a gradient on $y$. This gradient-based overestimation could potentially not overestimate $f(x)$, so the next step is to prove that it indeed works.

\begin{figure}[t]
    \centering
    \begin{subfigure}{0.4\textwidth}
        \centering
        \begin{overpic}[abs,unit=1mm,width=0.82\textwidth]{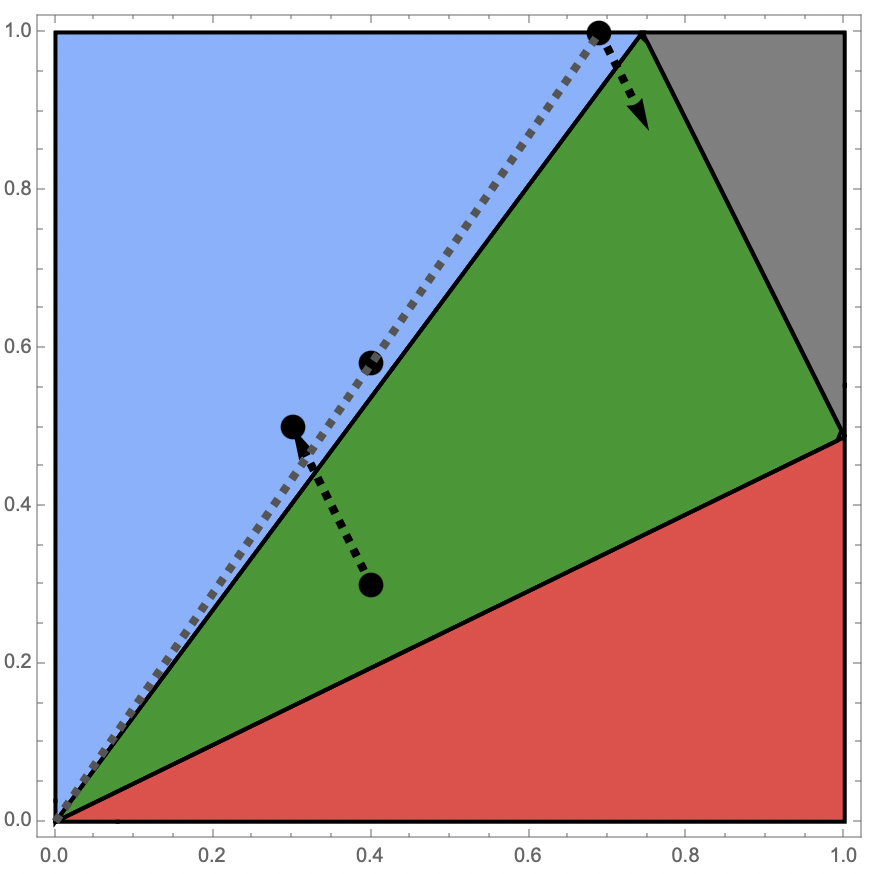}
\put(17,11){\color{black}\large$x$}
\put(3.5,12){\rotatebox{55}{\large$x+\alpha \tilde{w}$}}
\put(17,20){\color{black}\large$y$}
\put(19,34){\color{black}\large$y/y_i$}
\end{overpic}
        \caption{Illustration in the case $y/y_i$ is such that $g(y/y_i) = f(y/y_i)$.}
    \end{subfigure}%
    \hspace{1cm}
    \begin{subfigure}{0.4\textwidth}
        \centering
        \begin{overpic}[abs,unit=1mm,width=0.82\textwidth]{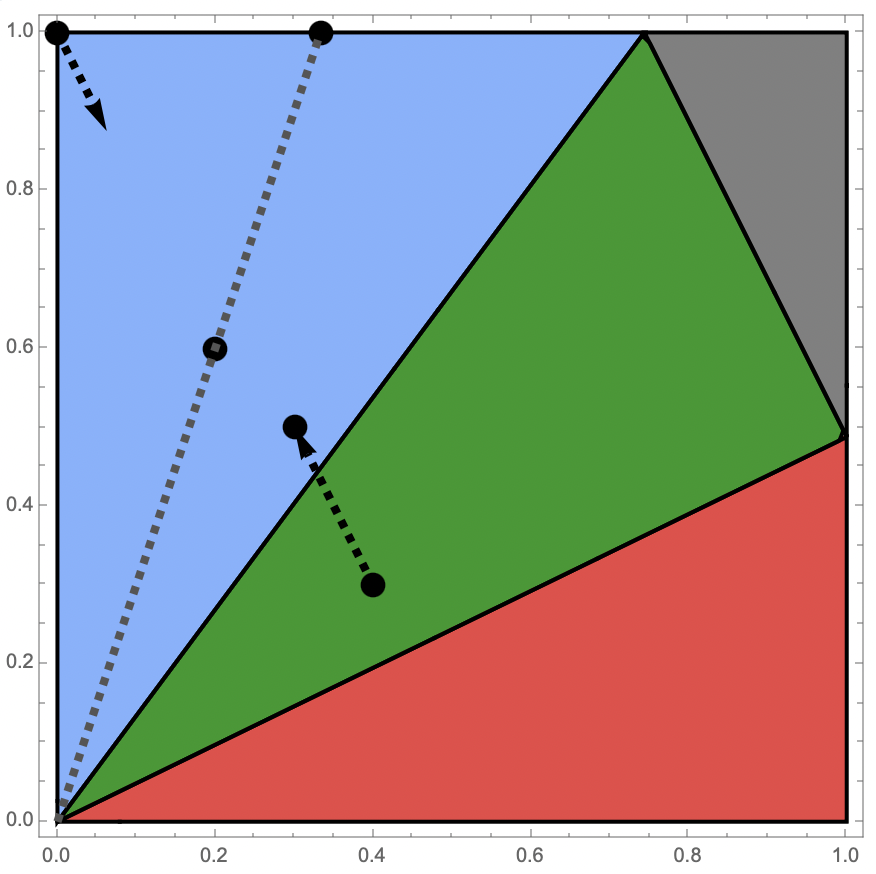}
\put(17,11){\color{black}\large$x$}
\put(12,20){\rotatebox{55}{\large$x+\alpha \tilde{w}$}}
\put(7,25){\color{black}\large$y$}
\put(14,34){\color{black}\large$y/y_i$}
\put(-2,35){\color{black}\large$e_i$}
\end{overpic}
        \caption{Illustration in the case $y/y_i$ is such that $g(y/y_i) \neq f(y/y_i)$.}
    \end{subfigure}
    \caption{Proof strategy for concavity of $g$. $R_i$ is the light blue region. In both cases, the dashed arrows point in opposite directions and are orthogonal to $w$.} \label{fig:proofidea}
\end{figure}

Let us first show that such a $\tilde w$ always exists. While in two dimensions it is fairly simple to see, in higher dimensions, one needs to be more careful.

\vskip .5cm
\begin{lemma}\label{lemma:omegatilde}
Let $\tilde{y} \in R_i$ for some $i$ such that $\tilde{y}_i = 1$ and $\tilde{y}_j < 1$ for $j\neq i$, and $x\in (0,1)^n \setminus (R_f \cup R_i)$. There exists $\tilde{w}\in \mathbb{R}^n$ such that
(i) $\tilde{w}^\top w = 0$,
(ii) $x+\alpha \tilde w \in R_i$ for some $\alpha \geq 0$, and
(iii) $\tilde{y} - \varepsilon \tilde w \in R_i$ for some $\varepsilon > 0$.
\end{lemma}

\begin{proof}
    If there exists such $\tilde{y}$, then $e_i\in R_i$ by Lemma \ref{lemma:eiinRi}. We proceed with a case distinction.
    \paragraph{Case 1.} If $w^\top e_i \geq w^\top x $, then we consider $\lambda \leq 1$ such that $w^\top (\lambda e_i) = w^\top x $ and define
        \( \tilde w = \lambda e_i - x\)
        and $\alpha = 1$. Then, $\tilde{w}^\top w = 0$ by construction and since $w^\top x + b < \hat{z}$ we have that $x + \alpha \tilde w = \lambda e_i \in R_i$.

        Let us show $\tilde{y} - \varepsilon \tilde w \in R_i$. Note that 
        \(w_ix_i < w^\top x = \lambda w_i \, \Rightarrow x_i < \lambda. \) 
        The $i$-th component of $\tilde{y} - \varepsilon \tilde w$ is
        \(1-\varepsilon (\lambda - x_i) \)
        and the rest is
        \(\tilde{y}_j + \varepsilon x_j. \)
        Since $\lambda > x_i$ and $\tilde{y}_j < 1$ for $j\neq i$, we can ensure that for $\varepsilon$ small $\tilde{y} - \varepsilon \tilde w \in [0,1]^n$, and that the $i$-th component is strictly larger than the rest. Finally,
        \[w^\top \frac{(\tilde{y} - \varepsilon \tilde{w})}{1-\varepsilon (\lambda - x_i)}  -  (\hat{z} - b) \xrightarrow{\varepsilon \to 0}  w^\top \tilde{y} - (\hat{z} - b) < 0.\]
        Thus, for $\varepsilon$ small enough, 
        \(w^\top (\tilde{y} - \varepsilon \tilde{w})/\|\tilde{y} - \varepsilon \tilde{w}\|_\infty  <  \hat{z} - b.\)
        We conclude $\tilde{y} - \varepsilon \tilde{w} \in R_i$.

       \paragraph{Case 2.} If $w^\top e_i < w^\top x $, then we consider $\lambda > 1$ such that $w^\top (\lambda e_i) = w^\top x $ and, as before, define
        \( \tilde w = \lambda e_i - x.\)
        Take
        \(\alpha = (1-x_i)/(\lambda - x_i).\)
        Since $x_i < 1 < \lambda$, $\alpha\in (0,1)$.
        In this case, $\tilde{w}^\top w = 0$ again follows by construction. To show that $x + \alpha \tilde{w}\in R_i$, note that
        \(
            (x + \alpha \tilde{w})_i = 1
            \)
            and for $j\neq i$,
            \(
            (x + \alpha \tilde{w})_j = (1-\alpha) x_j.
        \)
        Therefore, $x+ \alpha \tilde w \in [0,1]^n$, $\|x + \alpha \tilde{w}\|_\infty = 1$, and  
        \(w^\top (x + \alpha \tilde{w}) = w^\top x < \hat{z} - b.\)
        Thus, $x + \alpha \tilde{w} \in R_i$.

        Finally, and similarly to the previous case, the $i$-th component of $\tilde{y} - \varepsilon \tilde{w}$ is
        \(1-\varepsilon (\lambda - x_i) \)
        and the rest is
        \(\tilde{y}_j + \varepsilon x_j. \)
        In this case, we also have $\lambda > x_i$, and since $\tilde{y}_j < 1$ for $j\neq i$, we can ensure that for $\varepsilon$ small, the $i$-th component of $\tilde{y} - \varepsilon \tilde w$ is the largest. Finally,
        \[w^\top \frac{(\tilde{y} - \varepsilon \tilde{w})}{1-\varepsilon (\lambda - x_i)}  -  (\hat{z} - b) \xrightarrow{\varepsilon \to 0}  w^\top \tilde{y} - (\hat{z} - b) < 0.\]
        Thus, for $\varepsilon$ small enough, 
        \(w^\top (\tilde{y} - \varepsilon \tilde{w})/\|\tilde{y} - \varepsilon \tilde{w}\|_\infty  <  \hat{z} - b. \)
        i.e., $\tilde{y} - \varepsilon \tilde{w} \in R_i$.
\end{proof}

The last piece of information we need for establishing concavity is the gradient of $g$ on the interior of a region $R_i$. Let us define
\[g_{ave,-i} = \conc(f_{-i},[0,1]^{n-1}) \]
Assuming $g_{ave,-i}$ is differentiable in $x_{-i}/x_i$:
\begin{subequations}\label{eq:gradientofg}
\begin{align}
\frac{\partial }{\partial x_j} g(x) &= \frac{\partial}{\partial x_j} g_{ave,-i}(x_{-i}/x_i) && j\neq i\\
\frac{\partial }{\partial x_i} g(x) &= g_{ave,-i}(x_{-i}/x_i) - \nabla g_{ave,-i}(x_{-i}/x_i)^\top \left( \frac{x_{-i}}{x_i} \right) 
\end{align}
\end{subequations}

\begin{lemma}
Inequality \eqref{eq:gradientineq} holds for every $y$ in a dense subset of $R_i$ and every $x \in (0,1)^n \setminus (R_i\cup R_f)$.
\end{lemma}
\begin{proof}

 We consider $y \in \inte(R_i)$ such that $g_{ave,-i}$ is differentiable in $y_{-i}/y_i$. This induces a dense subset of $\inte(R_i)$ since $g_{ave,-i}$ is differentiable almost everywhere.
 Moreover, $g$ is differentiable in $y$.
 In this case, $1 > y_i > y_j > 0$ for every $j\neq i$. We take $\tilde \omega$ as in Lemma \ref{lemma:omegatilde} (we leave which $\tilde y$ we use ambiguous for now since it depends on a case distinction below). Then, using $w^\top \tilde w = 0$ and that $g$ is a concave overestimator on $R_i$:
    \[
        f(x) = f(x + \alpha \tilde w) \leq g(x + \alpha \tilde w) \leq g(y) + \nabla g(y)^\top (x + \alpha \tilde w - y).
    \]
    Thus, it suffices to show $\nabla g(y)^\top    \tilde w \leq 0$. 
    By the inductive hypothesis, $\conc(f_{-i} - f_{-i}(0),[0,1]^{n-1})$ can be described using a function of the form of $g$. Following the notation preceding this lemma, we have
    \(\conc(f_{-i} - f_{-i}(0),[0,1]^{n-1}) = g_{ave,-i} - f_{-i}(0).\)
    We use Lemma \ref{lemma:linearorf} with $g_{ave,-i} - f_{-i}(0)$ to get that $y_{-i}/y_i$ must be in one of two cases.

    \paragraph{Case 1 in Lemma \ref{lemma:linearorf}.} 
    In this case the envelope is tight at $y_{-i}/y_i$. This means $g_{ave,-i}(y_{-i}/y_i) - f_{-i}(0) = f_{-i}(y_{-i}/y_i) - f_{-i}(0)$.
    Recall that 
        \[ f_{-i}(y_{-i}/y_i) - f_{-i}(0) =  f(y/y_i) - f(e_i).\]
        For this case, we choose $\tilde w$ for $\tilde{y} := y/y_i$ in Lemma \ref{lemma:omegatilde}. With this, we move in $-\tilde w$ from $y/y_i$: this corresponds to the case of Figure \ref{fig:proofidea}(a). This yields,
        \begin{align*}
            f(y/y_i) 
            &\leq g(y/y_i - \varepsilon \tilde w) && \text{($g$ overest. in $R_i$ and $w \perp \tilde w$)}\\
            &\leq g(y) + \nabla g(y)^\top (y/y_i - \varepsilon \tilde w - y) && \text{($g$ concave in $R_i$)}\\
            &= g(y/y_i) - \varepsilon \nabla g(y)^\top \tilde w  && \text{($g$ linear on the rays in $R_i$)} \\
            &= g_{ave,-i}(y_{-i}/y_i) - \varepsilon \nabla g(y)^\top \tilde w  && \text{(definition of $g$)} \\
            &= f_{-i}(y_{-i}/y_i) - \varepsilon \nabla g(y)^\top \tilde w  &&\text{(case definition)} 
        \end{align*}
        Since $f_{-i}(y_{-i}/y_i) = f(y/y_i)$ we obtain $\nabla g(y)^\top \tilde w \leq 0$.

    \paragraph{Case 2 in Lemma \ref{lemma:linearorf}.} Since we are assuming $g_{ave,-i}$ is differentiable in $y_{-i}/y_i$, this case can be stated as  
    \[g_{ave,-i}(y_{-i}/y_i) - f_{-i}(0)= \nabla g_{ave,-i}(y_{-i}/y_i)^\top y_{-i}/y_i.\] 
    Note that in the previous case, we heavily relied on the envelope being tight for $y/y_i$, which may not be the case here. However, since $g_{ave,-i}$ is linear in a ray, we can move linearly to a point that is always tight: $e_i$. 
    Therefore, for this case, we choose $\tilde w$ for $\tilde{y}:= e_i$ in Lemma \ref{lemma:omegatilde} and move in $-\tilde w$: this corresponds to the case of Figure \ref{fig:proofidea}(b). This yields,
    \[
            f(e_i) 
            = f(e_i - \varepsilon \tilde w)  
            \leq g(e_i - \varepsilon \tilde w) 
            \leq g(y) + \nabla g(y)^\top (e_i - \varepsilon \tilde w - y).
     \]
    Let us analyze the product $\nabla g(y)^\top (e_i  - y)$. From \eqref{eq:gradientofg}
\begin{align*}
    \nabla g(y)^\top (e_i  - y) &= -\nabla g_{ave,-i} (y_{-i}/y_i)^\top y_{-i} \\
    & \qquad + (1-y_i)\left(g_{ave,-i}(y_{-i}/y_i) - \nabla g_{ave,-i}(y_{-i}/y_i)^\top \left( \frac{y_{-i}}{y_i} \right) \right)\\
    &=  (1-y_i)g_{ave,-i}(y_{-i}/y_i) -  \nabla g_{ave,-i}(y_{-i}/y_i)^\top \left( \frac{y_{-i}}{y_i} \right) 
\end{align*}
Also recall that $g(y) = y_i g_{ave,-i}(y_{-i}/y_i)$, thus
\begin{align*}
            f( e_i ) 
            &\leq g_{ave,-i}(y_{-i}/y_i) -  \nabla g_{ave,-i}(y_{-i}/y_i)^\top \left( \frac{y_{-i}}{y_i} \right) - \varepsilon \nabla g(y)^\top \tilde w \\
            &= f_{-i}(0) - \varepsilon \nabla g(y)^\top \tilde w = f(e_i) - \varepsilon \nabla g(y)^\top \tilde w \qquad \text{(case definition)}
        \end{align*}
        This implies $ \nabla g(y)^\top \tilde w \leq 0$.
\end{proof}

\newpage

\section*{Acknowledgements}
We would like to thank Felipe Serrano and Juan Pablo Vielma for their helpful feedback on an early version of this article, as well as Joey Huchette, Mohit Tawarmalani, and Calvin Tsay for their valuable comments.
We would also like to thank the two anonymous reviewers for their feedback.
P. Carrasco and G. Muñoz were supported by the National Research and Development Agency of Chile (ANID) through the Fondecyt Grant 1231522.

\bibliography{sn-bibliography}

\newpage

\appendix

\section{Convex and S-shaped activations}\label{app:list}

In Table \ref{table:activations}, we list common convex and S-shaped activations.
The SiLU activation is not technically S-shaped since it is concave \emph{and then} convex. However, we can reflect the domain.

\begin{table}[t]
\centering
\caption{List of common convex and S-shaped activation functions} \label{table:activations}
\begin{tabular}{|l|l| l | l |}
\hline
\textbf{Name} & \textbf{Parameters} & \textbf{Type} & \textbf{Function $\sigma(x)$} \\ \hline
Rectifier Linear Unit (ReLU)   & & Convex & $\max\{0, x\}$       \\ \hline
Leaky ReLU & $0 < \varepsilon < 1$ & Convex &
$\begin{cases} 
x, & \text{if } x > 0 \\
\varepsilon x, & \text{if} x\leq 0 
\end{cases}$ \\ \hline
Maxtanh                    & & Convex & $\max\{x, \tanh(x)\}$  \\ \hline
Softplus                   & & Convex & $\log(1 + \exp(x))$  \\ \hline
Exponential Linear Unit (ELU) & $0< \alpha \leq 1 $ & Convex & 
$\begin{cases} 
x, & x > 0 \\
\alpha (\exp(x) - 1), & x \leq 0 
\end{cases}$ \\ \hline
\hline
Softsign                        & & S-shaped & $\displaystyle \frac{x}{1 + |x|}$  \\ \hline
Hyperbolic tangent               & & S-shaped &$\tanh(x)$           \\ \hline
Penalized hyperbolic tangent     & $0<\alpha < 1$ & S-shaped & 
$\begin{cases} 
\tanh(x), & x > 0 \\
\tanh(\alpha x), & x \leq 0 
\end{cases}$ \\ \hline
Sigmoid                         & & S-shaped & $1/(1 + \exp(-x))$ \\ \hline
Bipolar sigmoid                  & & S-shaped & $(1 - \exp(-x))/(1 + \exp(-x))$ \\ \hline
Exponential Linear Unit (ELU) & $ \alpha > 1$ & S-shaped & 
$\begin{cases} 
x, & x > 0 \\
\alpha (\exp(x) - 1), & x \leq 0 
\end{cases}$ \\ \hline
Scaled ELU (SELU) & $\begin{array}{l}\lambda = 1.0507\\ \alpha = 1.67326\end{array}$ & S-shaped & 
$\lambda \begin{cases} 
x, & x > 0 \\
\alpha (\exp(x) - 1), & x \leq 0 
\end{cases}$ \\ \hline
Sigmoid Linear Unit (SiLU) & & S-shaped* & $x/(1+\exp(-x))$ \\ \hline
\end{tabular}
\end{table}

\section{Proof of Lemma \ref{lemma:onedprops}}\label{app:oned}

The cases "Decreasing upper bound, case 1" and "Increasing lower bound" are direct.

For the ``Decreasing upper bound, case 2'', we consider $\tilde U \in [L,\hat z]$. We assume $L<\hat z$; otherwise, the result is trivial.
Let $\tilde z$ be the tie point in $[L, \tilde U]$. If $\tilde z = \tilde U$, we are done, thus we assume $\tilde z < \tilde U$, which implies $\sigma$ is concave in $[\tilde z, \tilde U]$ and the latter has nonempty interior.
Let $[m,M]\supseteq [\tilde z, \tilde U] $ be a maximal interval where $\sigma$ is concave. 
We claim that $M\leq \hat z$. Otherwise, $\sigma$ is concave on $[\tilde z, \hat z + \varepsilon]$ for some $\varepsilon>0$ and we get that

\[
\left\{ \begin{array}{ll}
     \conc(\sigma, [L, \tilde{U}] )(z) & \text{if } z\in [L, \tilde z) \\
     \sigma(z)& \text{if } z\in [ \tilde z, U]
\end{array}\right.
\]
is a valid concave overestimator of $\sigma$ on $[L,U]$, which would imply the tie point in $[L,U]$ is less or equal than $\tilde z < \hat z$. This implies $\sigma$ is not concave in $[M-\delta,M+\varepsilon]$ for any $\delta,\varepsilon>0$.
Additionally, the tie point in $[L, M]$ must be $\tilde z$.

By definition of $M$ and the STFE property, $\conc(\sigma, [m, M+\varepsilon])$ is an affine function in $[m,M]$ (there could potentially be a tie point in $[M,M+\varepsilon]$ since the function could have a concave portion in this interval): this affine function is tight at $m$ and at a point in $[M,M+\varepsilon]$.
Taking $\varepsilon\to 0$, we obtain an affine overestimator $\Lambda(z)$ of $\sigma$ in $[m, M]$ which is tight at the endpoints. Since $\sigma$ is concave in this interval, we conclude that it must be \emph{affine} in $[m,M]$. 
Note that $\Lambda(z)$ must define an overestimator of $\sigma$ on $[L, M]$, since we know that the tie point of $\sigma$ in $[L, M]$ is $\tilde z \in [m, M)$.  

If $m\leq L$ we are done, so we assume $m > L$. In this case, it must hold that $\tilde z = m$.
Moreover, there must exist $\tilde L \in (L, m)$ such that 
\(\sigma(\tilde L) < \Lambda(\tilde L). \)
By the ``Increasing lower bound" property, the tie point in $[\tilde L, M]$ is $m$; it cannot be strictly smaller since $\sigma$ is not concave on $[m-\delta, M]$. 
This shows that $\conc(\sigma, [\tilde L, M])$ cannot be an affine function: if it were, it would have to be $\Lambda$ as $[m,M]$ has nonempty interior, but we know $\Lambda$ is not tight at $\tilde L$ and the envelopes are always tight on the endpoints.

To conclude, we consider $[\tilde L, M+\varepsilon]$. By the same arguments as above, $\conc(\sigma, [\tilde L, M+\varepsilon])$ must be an affine function in $[\tilde L, M]$ which is tight in $\tilde L$ and in some point in $[M, M+\varepsilon]$. Letting $\varepsilon\to 0$, we obtain an affine overestimator of $\sigma$ on $[\tilde L, M]$ which is tight at the endpoints. This is a contradiction.

\end{document}